\documentclass{amsart}
\usepackage[margin=1.5in]{geometry}

\usepackage{amsmath,amsthm, latexsym, amssymb,mathrsfs,  tikz-cd}
\usepackage{stackengine,scalerel, quiver, MnSymbol}
\usepackage[colorlinks]{hyperref}
\definecolor{dark-red}{rgb}{0.6,0,0}
\definecolor{dark-green}{rgb}{0,0.4,0}
\definecolor{medium-blue}{rgb}{0,0,0.5}
\hypersetup{
    colorlinks, linkcolor={dark-red},
    citecolor={dark-green}, urlcolor={medium-blue}
}

\newcommand{\Rep}{\mathrm{Rep}}

\newcommand{\BB}{\mr{BB}}

\newcommand{\Fil}{\mr{Fil}}

\newcommand{\gr}{\mr{gr}}
\newcommand{\Gr}{\mathrm{Gr}}

\newcommand{\Fl}{\mr{Fl}}

\newcommand{\Spd}{\mathrm{Spd}}
\newcommand{\GL}{\mathrm{GL}}
\newcommand{\HT}{\mr{HT}}

\newcommand{\mc}[1]{\mathcal{#1}}
\newcommand{\mbb}[1]{\mathbb{#1}}
\newcommand{\mr}[1]{\mathrm{#1}}
\newcommand{\mf}[1]{\mathfrak{#1}}

\newcommand{\et}{\mathrm{\acute{e}t}}

\newcommand{\dR}{\mathrm{dR}}

\newcommand{\adm}{\mathrm{adm}}

\newcommand{\univ}{\mathrm{univ}}

\DeclareMathOperator{\Lie}{Lie}

\DeclareMathOperator{\Hom}{Hom}

\newcommand{\ul}[1]{\underline{#1}}

\newcommand{\Hdg}{\mr{Hdg}}

\numberwithin{equation}{subsection}
\numberwithin{equation}{subsubsection}

\theoremstyle{plain}

\newtheorem{maintheorem}{Theorem}

\newtheorem*{theorem*}{Theorem}

\newtheorem{lemma}[subsubsection]{Lemma}

\theoremstyle{definition}

\newtheorem{example}[subsubsection]{Example}

\newtheorem{remark}[subsubsection]{Remark}

\newcommand{\lf}{\mathrm{lf}}
\newcommand{\lfid}{{\diamond_{\lf}}}

\newcommand{\proet}{\mathrm{pro\acute{e}t}}

\title[The geom. Sen morphism is the unique lift of the Kod.-Spen. morphism]{The geometric Sen morphism is the unique lift of the Kodaira--Spencer morphism}

\author{Sean Howe}
\begin{document}

\begin{abstract} 
We show that the geometric Sen morphism of a de Rham torsor over a smooth rigid analytic variety over a $p$-adic field is the unique lift, along a natural map, of the Kodaira--Spencer morphism of the associated filtered torsor with integrable connection. This extends previous computations in the minuscule case, and implies that the geometric Sen morphism is the derivative of the lattice Hodge period map. The computation applies, in particular, to non-minuscule period domains generalizing local Shimura varieties, furnishing new examples of towers satisfying He's stalkwise perfectoidness. 
\end{abstract}

\maketitle

\tableofcontents

\section{Introduction}
This geometric Sen morphism of Pan \cite{Pan.OnLocallyAnalyticVectorsOfTheCompletedCohomologyOfModularCurves} and Rodriguez Camargo \cite{RodriguezCamargo.GeometricSenTheoryOverRigidAnalyticSpaces} is a canonical twisted Higgs field associated to a $p$-adic Lie group torsor over a rigid analytic variety. Since its introduction, it has played an important role in $p$-adic geometry and its applications to automorphic forms \cite{Pan.OnLocallyAnalyticVectorsOfTheCompletedCohomologyOfModularCurves, Pan.OnLocallyAnalyticVectorsOfTheCompletedCohomologyOfModularCurvesII, RodriguezCamargo.LocallyAnalyticCompletedCohomology, He.PerfectoidnessViaSenTheory, DospinescuRodriguezCamargo.JacquetLanglands}. For these applications, it is often necessary to have a description of the geometric Sen morphism in terms of more classical data. Such descriptions have been computed previously for infinite level global Shimura varieties \cite{RodriguezCamargo.LocallyAnalyticCompletedCohomology} (building on the case of the modular curve in \cite{Pan.OnLocallyAnalyticVectorsOfTheCompletedCohomologyOfModularCurves}), local Shimura varieties \cite{DospinescuRodriguezCamargo.JacquetLanglands}, and compact $p$-adic Lie group torsors over abelian varieties \cite{BellovinCaiHowe.CharacterizingPerfectoidCoversOfAbelianVarieties}. The main result of the present work (Theorem \ref{theorem.lift}) extends these prior computations by giving an explicit description of the geometric Sen morphism for any de Rham $p$-adic Lie group torsor in terms of the associated filtered analytic torsor with integrable connection. In particular, this result applies to the local systems coming from general smooth proper families of rigid analytic varieties where the Hodge cocharacter may no longer be minuscule (i.e., beyond the Shimura setting) and to non-miniscule rigid analytic crystalline period domains that generalize local Shimura varieties (Example \ref{ex.period-domains}). 

 Explicitly, Theorem \ref{theorem.lift} says that the geometric Sen morphism is the unique lift (along a natural map) of the Kodaira--Spencer morphism. This has a natural interpretation in terms of period maps (\S\ref{sss.period-maps}): the Kodaira--Spencer morphism is best interpreted as the derivative of the Hodge period map in the canonical flat direction, and in $p$-adic geometry Hodge period maps lift canonically to lattice Hodge period maps with target the Schubert split
$B^+_\dR$-affine Grassmannian. Our result then implies that the geometric Sen morphism is the derivative of this lattice lift (this is made precise using the inscribed $v$-sheaves of \cite{Howe.InscriptionTwistorsAndPAdicPeriods}). 

One consequence of Theorem \ref{theorem.lift} is that injectivity of geometric Sen can be checked at the level of Kodaira-Spencer. This has implications for perfectoidness; in particular, it furnishes many new examples where He's stalkwise perfectoidness criterion applies (\S\ref{sss.perfectoidness}).

\subsection{Statement of the main result}
Before discussing these examples and applications in more detail, we want to give a precise statement of our result. Thus, we first set up some notation about $p$-adic Lie torsors and filtered torsors with integrable connection. To that end, let $G/\mathbb{Q}_p$ be a connected linear algebraic group, let $K \subseteq G(\mathbb{Q}_p)$ be an open subgroup, and let $\mf{g}:=\Lie K = \Lie G(\mathbb{Q}_p)$. Let $L$ be a $p$-adic field (that is, a complete discretely valued extension of $\mathbb{Q}_p$ with perfect residue field), and let $S/L$ be a smooth rigid analytic variety. 

\subsubsection{The geometric Sen morphism}
  From $\tilde{S}/S$ a pro\'{e}tale $K$-torsor, \cite[Theorem 1.0.4] {RodriguezCamargo.GeometricSenTheoryOverRigidAnalyticSpaces} constructs a canonical twisted Higgs field
\[ \theta_{\tilde{S}} \in \left(\Omega^1_{S/C}\otimes_{\mathcal{O}} \left(\tilde{\mathfrak{g}}  \otimes_{\ul{\mathbb{Q}_p}} {\widehat{\mathcal{O}}}(-1)\right) \right)(S), \]
where $\tilde{\mathfrak{g}}=\tilde{S} \times^{K} \ul{\mathfrak{g}}$ is the form of the constant local system $\ul{\mathfrak{g}}$  obtained by twisting by $\tilde{S}$ along the adjoint action of $K$ on $\mathfrak{g}$. Dually, we view $\theta_{\tilde{S}}$ as a morphism on $S_\proet$,
\[ \kappa_{\tilde{S}}: T_{S/C} \otimes_{\mathcal{O}} \hat{\mathcal{O}} \rightarrow \tilde{\mathfrak{g}}\otimes_{\ul{\mbb{Q}_p}} \hat{\mathcal{O}}(-1). \]

\subsubsection{de Rham torsors and the Kodaira--Spencer morphism}
A pro\'{e}tale $K$-torsor $\tilde{S}/S$ is called de Rham if, for every $V \in \Rep_{\mathbb{Q}_p} G$, the associated $\mathbb{Q}_p$-local system $\tilde{K} \times^K V$ is de Rham in the sense of \cite{Scholze.pAdicHodgeTheoryForRigidAnalyticVarieties} (it is equivalent to check this for one faithful representation and, by \cite{LiuZhu.RigidityAndARiemannHilbertCorrespondenceForpAdicLocalSystems}, it even suffices to check that the pullback to a single classical point in each connected component is a de Rham Galois representation). Associated to any de Rham torsor $\tilde{S}$, we have a filtered $G$-torsor with integrable connection satisfying Griffiths transversality, $\tilde{S}_\dR/S$. We write $\tilde{\mf{g}}_\dR=\tilde{S}_\dR \times^G \Lie G$, which is the filtered vector bundle with integrable connection satisfying Griffiths transversality associated to the de Rham  $\mathbb{Q}_p$-local system $\tilde{\mathfrak{g}}$. It has a natural Hodge filtration, and the Kodaira--Spencer morphism for $\tilde{S}_\dR$ is a map
\[ \kappa_{\tilde{S}_\dR}: T_{S/C} \rightarrow  \gr^{-1}_\Hdg (\tilde{\mf{g}}_\dR) \]
(that it lands inside of $\gr^{-1}_\Hdg (\tilde{\mf{g}}_\dR)$ rather than just  $\tilde{\mf{g}}_\dR/\Fil^0 \tilde{\mf{g}}_\dR$ \emph{is} Griffiths transversality). 

\subsubsection{The comparison}
For $\tilde{S}/S$ a de Rham $K$-torsor, there is a Hodge-Tate filtration on $(\tilde{\mf{g}} \otimes_{\ul{\mbb{Q}_p}} \hat{\mathcal{O}})$ and the Hodge-Tate comparison gives a canonical isomorphism 
\[ \left(\gr^1_\HT (\tilde{\mf{g}} \otimes_{\ul{\mbb{Q}_p}} \hat{\mathcal{O}})\right)(-1)= \left(\gr^{-1}_\Hdg (\tilde{\mf{g}}_\dR)\right) \otimes_{\mathcal{O}} \hat{\mathcal{O}}. \]
On the other hand, we have the natural quotient map $\Fil^1_\HT(\tilde{\mf{g}} \otimes_{\ul{\mbb{Q}_p}} \hat{\mathcal{O}}) \twoheadrightarrow \gr^{1}_\HT(\tilde{\mf{g}} \otimes_{\ul{\mbb{Q}_p}} \hat{\mathcal{O}})$. Composing these two maps, we obtain a natural surjection
\begin{equation}\label{eq.nat-surj}q_{\tilde{S}}: \left(\Fil^1_\HT(\tilde{\mf{g}} \otimes_{\ul{\mbb{Q}_p}} \hat{\mathcal{O}})\right)(-1) \twoheadrightarrow \left(\gr^{-1}_\Hdg (\tilde{\mf{g}}_\dR)\right) \otimes_{\mathcal{O}} \hat{\mathcal{O}}.\end{equation}

Our main result compares $\kappa_{\tilde{S}}$ and $\kappa_{\tilde{S}_\dR}$ using this map: 
\begin{maintheorem}\label{theorem.lift}
    Let $L$ be a $p$-adic field, let $S/L$ be a smooth rigid analytic variety, let $G/\mathbb{Q}_p$ be a connected linear algebraic group, and let $K \leq G(\mathbb{Q}_p)$ be an open subgroup. For any de Rham $K$-torsor $\tilde{S}/S$, $\kappa_{\tilde{S}}$ factors through $\left(\Fil^1_\HT(\tilde{\mf{g}} \otimes_{\ul{\mbb{Q}_p}} \hat{\mathcal{O}})\right)(-1)$ and is the unique map 
    \[ T_{S/C} \otimes_{\mathcal{O}} \hat{\mathcal{O}} \rightarrow \left(\Fil^1_\HT(\tilde{\mf{g}} \otimes_{\ul{\mbb{Q}_p}} \hat{\mathcal{O}})\right)(-1) \]
    whose composition with $q_{\tilde{S}}$ is equal to $\kappa_{\tilde{S}_\dR}.$ 
\end{maintheorem}

\begin{remark}\label{remark.minuscule}
    In particular, if $\Fil^2_\HT(\tilde{\mf{g}} \otimes_{\ul{\mbb{Q}_p}} \hat{\mathcal{O}})=0$, then 
    \[ \left(\Fil^1_\HT(\tilde{\mf{g}} \otimes_{\ul{\mbb{Q}_p}} \hat{\mathcal{O}})\right)(-1) = \left(\gr^1_\HT(\tilde{\mf{g}} \otimes_{\ul{\mbb{Q}_p}} \hat{\mathcal{O}})\right)(-1) \]
    and in this case Theorem \ref{theorem.lift} says $\kappa_{\tilde{S}}=\kappa_{\tilde{S}_\dR}.$ If $G$ is reductive and there is a Hodge cocharacter (e.g. if $S$ is geometrically connected), then this is precisely the case in which that cocharacter is minuscule --- this is the situation one is in when considering local and global Shimura varieties. In general, the uniqueness statement means that even when $\Fil^2_\HT(\tilde{\mf{g}} \otimes_{\ul{\mbb{Q}_p}} \hat{\mathcal{O}}) \neq 0$, the Kodaira-Spencer morphism uniquely determines the geometric Sen morphism. 
\end{remark}

\subsubsection{Sketch of proof}
That such a lift is unique if it exists is a straightforward consideration of Hodge-Tate weights and the vanishing of $\nu_*\left( \hat{\mathcal{O}}(k)\right)$ for $k<0$ established in \cite{Scholze.pAdicHodgeTheoryForRigidAnalyticVarieties}, where $\nu$ is the natural map $S_\proet \rightarrow S_\et$. To show that geometric Sen does indeed lift Kodaira--Spencer, since both are functorial it suffices to treat the case of $G=\GL_n$. In this case it can be extracted directly from a comparison with the $p$-adic Simpson correspondence of \cite{LiuZhu.RigidityAndARiemannHilbertCorrespondenceForpAdicLocalSystems}. This same compatibility is used already in the proofs for local and global Shimura varieties, thus, the main innovation in Theorem \ref{theorem.lift} was perhaps just to find a natural statement. 

\subsection{Examples and applications}

For canonical torsors over Shimura varieties, Theorem~\ref{theorem.lift} was established in \cite{RodriguezCamargo.LocallyAnalyticCompletedCohomology} (building on the case of the modular curve treated in   \cite{Pan.OnLocallyAnalyticVectorsOfTheCompletedCohomologyOfModularCurves}). The case of canonical torsors over local Shimura varieties was established in \cite{DospinescuRodriguezCamargo.JacquetLanglands}. Both are in the minuscule setting as in Remark \ref{remark.minuscule}; below we give one more example (Example \ref{example.av}) where we recover a prior computation in the minuscule setting, then discuss some new examples in the non-minuscule setting. In particular, Example \ref{ex.period-domains} generalizes the case of local Shimura varieties to non-miniscule flat crystalline period domains. In \S\ref{sss.perfectoidness} we explain how these examples relate to conjectures and results on perfectoidness of torsors. Finally, in \S\ref{sss.period-maps} we explain the connection with derivatives of period maps. 

\begin{example}[Abelian varieties]\label{example.av}
Let $L$ be a $p$-adic field, let $C=\overline{L}^\wedge$, and let $A/L$ be an abelian variety. The tower $(A)_{n \in \mathbb{N}}$ with transition maps multiplication by $p$ is a torsor over $A$ for the $p$-adic Tate module\footnote{This is not strictly within the purview of Theorem \ref{theorem.lift} since $T_p A$ over $L$ is not a $p$-adic Lie group but rather an arithmetically twisted $p$-adic Lie group. Nonetheless, there is still a geometric Sen morphism because, for $g=\dim A$, $(T_p A)_{C}\cong \ul{\mathbb{Z}_p}^{2g}$ is a $p$-adic Lie group, and the geometric Sen morphism can be constructed purely over $C$. To make the following argument precise one can instead work with a $\GL_{2g}(\mathbb{Z}_p) \ltimes \mathbb{Z}_p^{2g}$-torsor by taking also trivializations of $T_p A$ (so that the $\GL_{2g}(\mathbb{Z}_p)$ part of the torsor is purely arithmetic).} $T_p A$. The associated torsor on the de Rham side is the $H_{1,\dR}(A)$-torsor of splittings of the relative de Rham homology sequence over $A$
\[ 0 \rightarrow H_{1,\dR}(A) \otimes_{L} \mathcal{O}_A \rightarrow \mathcal{H}_{1,\dR}(A, \{a, 0\}) \rightarrow  \mathcal{O}_A \rightarrow 0 \]
associated to the map of varieties over $A$, 
\[ \{\ast_1, \ast_2\} \times A \rightarrow A \times A / A,\; (\ast_1, a) \mapsto (0,a) \textrm{ and } (\ast_2, a)\mapsto (a, a). \]
By a complex analytic computation (passing to a finitely generated extension of $\mathbb{Q}$ over which $A$ is defined and then base changing to $\mathbb{C}$), one finds that, above a point $a \in A$, the Kodaira-Spencer map on $\mathcal{H}_{1,\dR}(A, \{a,0\}))$ sends $t \in T_{a} A=\Lie A$ to the element of $\gr^{-1} \mathrm{End}( \mathcal{H}_{1,\dR}(A, \{a,0\}))$ represented by the composition of projection from $\mathcal{H}_{1,\dR}(A, \{a,0\})$ onto $\mathcal{O}_A$ with the map from $\mathcal{O}_A$ sending $1$ to 
\[ t \in \Lie A=H_{1,\dR}(A)/\Fil^0 H_{1,\dR}(A) \subseteq H_{1,\dR}(A, \{a,0\})/\Fil^0 H_{1,\dR}(A, \{a,0\}). \]
Applying Theorem \ref{theorem.lift}, one finds that the geometric Sen morphism is the map induced by the inclusion of the Hodge-Tate filtration $\Lie A \subseteq T_p A \otimes C(-1)$, recovering in this case \cite[Theorem 1.10] {BellovinCaiHowe.CharacterizingPerfectoidCoversOfAbelianVarieties} (but the argument of \cite{{BellovinCaiHowe.CharacterizingPerfectoidCoversOfAbelianVarieties}} also gives the result also for abelian varieties over $C$ that are not defined over a $p$-adic subfield!). Note that \cite[Theorem 1.10]{BellovinCaiHowe.CharacterizingPerfectoidCoversOfAbelianVarieties} also computes the geometric Sen morphism for the analogous cover of a $p$-divisible rigid analytic group; Theorem \ref{theorem.lift} can also be used to recover this computation for a $p$-divisible rigid analytic group $G$ defined over $L$ --- here, the relevant comparison is between the Tate and Dieudonn\'{e} modules of the universal extension of  $\mathbb{Q}_p/\mathbb{Z}_p$ by $G$ over $G$ constructed as in \cite[\S4]{HoweMorrowEtAl.TheConjugateUniformizationVia1Motives}. 
\end{example}

\begin{example}
A direct summand of the relative $p$-adic cohomology local system of a smooth proper family is a de Rham local system. If the Hodge numbers for a member of this family are not concentrated in two neighboring indices then the Hodge cocharacter is non-minuscule and the computation of the geometric Sen morphism in Theorem \ref{theorem.lift} is new (though in some cases can be recovered from previous computations for global and local Shimura varieties if the Hodge cocharacter is minuscule as a cocharacter of the Mumford-Tate group or local Mumford-Tate group). Explicit examples can be found, e.g., by taking the middle cohomology of a family of hypersurfaces of large degree in $\mathbb{P}^n$, $n \geq 3$.  
\end{example}

\begin{example}
Typically, the Hodge-Tate cocharacter of a de Rham $p$-adic Galois representation will not be minuscule even as a cocharacter of the Zariski closure of its image. In particular, when the closure is reductive, any de Rham torsor whose specialization at a classical point recovers such a representation will have non-minuscule Hodge cocharacter and thus cannot be constructed by pullback from a canonical torsor on a local or global Shimura variety. Thus the computation of Theorem \ref{theorem.lift} is completely new in these cases.  
\end{example}

\begin{example}\label{ex.period-domains}  For $G/\mathbb{Q}_p$ a connected linear algebraic group, we write $\Fl_G$ for the flag variety parameterizing filtrations on the trivial $G$-torsor. There is a natural Bialynicki-Birula map $\Gr_G^{\mr{split}} \rightarrow \Fl_G^\diamond$, where the domain $\Gr_G^{\mr{split}}$ is the disjoint union of the Schubert cells in the $B^+_\dR$-affine Grassmannian $\Gr_G$ parameterizing lattices on the trivial $G(\mathbb{B}^+_\dR)$-torsor.

Let $L$ be a $p$-adic field with residue field $\kappa$ and let $L_0=W(\kappa)[1/p]$. We base change all of the spaces in question to $L$. Then, by \cite[Theorem 5.0.4-(4)]{HoweKlevdal.AdmissiblePairsAndpAdicHodgeStructuresIITheBiAnalyticAxLindemannTheorem}, for $S/L$ a smooth rigid analytic variety, $\Hom_{\Spd L}(S^\diamond, \Gr_{G, L}^{\mr{split}})$ is canonically identified with the set of maps of rigid analytic varieties $f: S \rightarrow \Fl_G$ satisfying Griffiths transversality (i.e. such that $df$ factors through $f^*\gr^{-1}_{\univ} (\mf{g} \otimes \mathcal{O}_{\Fl_G}) \subseteq f^* T_{\Fl_G}=f^*(\mf{g} \otimes \mathcal{O}_{\Fl_G} / \Fil^{0}_{\univ}(\mf{g} \otimes \mathcal{O}_{\Fl_G} ) )$).

Now, for $b \in G(L_0)$, we have an open $b$-admissible locus $\Gr_G^{\mr{split}, b-\adm} \subseteq \Gr_G^{\mr{split}}$, with a canonical $G(\mathbb{Q}_p)$-torsor $\mathbb{U}$ over it (this is the universal flat crystalline $G(\mathbb{Q}_p)$-torsor for the $G$-isocrystal $b$). In particular, for any map $S \rightarrow \Fl_G$ atisfying Griffiths transversality, we have the associated lift $\tilde{f}: S \rightarrow \Gr_{G,L}^{\mr{split}}$ and an open subset $S^{b-\adm}:=\tilde{f}^{-1}(\Gr_G^{\mr{split}, b-\adm})$. Over $S^{b-\adm}$ we have the de Rham $G(\mathbb{Q}_p)$-torsor $f^*\mathbb{U}$, and by construction the associated Kodaira-Spencer morphism is identified with $df$; Theorem \ref{theorem.lift} then pins down its geometric Sen morphism. In particular, if we restrict to a component $\Fl_{[\mu]}$ of $\Fl_G$ corresponding to a geometric conjugacy class of miniscule cocharacters $[\mu]$, then we can take $S=\Fl_{[\mu]}$, and in this case we recover the computation of the geometric Sen morphism in the case of local Shimura varieties. However, just as in complex Hodge theory, there are also many smooth rigid analytic subvarieties living inside of non-minuscule components but still satisfying Griffiths transversality, and these provide  interesting non-miniscule rigid analytic period domains where our computation of the geometric Sen morphism applies and is new. 
\end{example}

\subsubsection{Perfectoidness}\label{sss.perfectoidness}
For $S/C$ a smooth rigid analytic variety,  $K$ a $p$-adic Lie group, and $\tilde{S}/S$ a pro-\'{e}tale $K$-torsor, \cite[Conjecture 3.3.5]{RodriguezCamargo.GeometricSenTheoryOverRigidAnalyticSpaces} predicts that $\tilde{S}_C$ is perfectoid if the geometric Sen morphism $\kappa_{\tilde{S}}$ is injective at every geometric point. This conjecture is known if $S$ is a semi-abelian variety and $K$ is compact by \cite[Theorem 1.10]{BellovinCaiHowe.CharacterizingPerfectoidCoversOfAbelianVarieties}, and there is also a large pile of evidence coming from the cases where infinite level local or global Shimura varieties are known to be perfectoid (starting with the Hodge-type case established in \cite{Scholze.OnTorsionInTheCohomologyOfLocallySymmetricVarieties}). 

When $S$ and $\tilde{S}/S$ both arise via base change from the $p$-adic field $L$, Tongmu He \cite[Theorem 1.17]{He.PerfectoidnessViaSenTheory} has established a stalkwise perfectoidness under the same injectivity hypothesies that is sufficient for many applications to cohomological vanishing. Theorem \ref{theorem.lift} implies that if the torsor is  de Rham and the Kodaira-Spencer morphism is injective at every geometric point then the geometric Sen morphism is also injective at every geometric point, thus \cite[Theorem 1.17]{He.PerfectoidnessViaSenTheory} applies in such a case (and we also expect global perfectoidness). In particular, this holds in the setting of Example \ref{ex.period-domains} precisely when the map $f$ to the flag variety is an immersion, and this yields many new examples where \cite[Theorem 1.17]{He.PerfectoidnessViaSenTheory} applies. 

\subsubsection{Derivatives of period maps}\label{sss.period-maps}
Inscribed $v$-sheaves, as introduced in \cite{Howe.InscriptionTwistorsAndPAdicPeriods}, provide a differential structure on top of the usual theory of diamonds and $v$-sheaves. 
The construction of \cite[\S7]{Howe.InscriptionTwistorsAndPAdicPeriods} associates to the filtered $G$-bundle with integrable connection $\tilde{S}_\dR$ an inscribed Hodge period map and inscribed lattice Hodge period map fitting in a commutative diagram
\[\begin{tikzcd}
	{S^{\lfid} } && {\Gr^{\mathrm{split}}_G/G(\overline{\mathbb{B}^+_\dR})} \\
	\\
	&& {\Fl_{G}^\lfid/G(\overline{\mathcal{O}})}
	\arrow["{\pi^+_\Hdg}"{description}, from=1-1, to=1-3]
	\arrow["{\pi_{\Hdg}}"{description}, from=1-1, to=3-3]
	\arrow["\BB"{description}, from=1-3, to=3-3]
\end{tikzcd}\]
Here $\Gr^\mr{split}_G$ is the disjoint union of Schubert cells in the (inscribed) $\mathbb{B}^+_\dR$-affine Grassmannian for $G$ and $\BB$ is the Bialynicki-Birula map. By \cite[Lemma 7.1.6]{Howe.InscriptionTwistorsAndPAdicPeriods}, the derivative of $\pi_{\Hdg}$ is (the inscibred extension of) the Kodaira-Spencer morphism $\kappa_{\tilde{S}_\dR}$, and the derivative of $\BB$ is naturally identified with (the inscribed extension of) the map $q_{\tilde{S}}$ as above. Because $d\pi_{\Hdg}= d\BB \circ d\pi_{\Hdg}^+$, we deduce from Theorem \ref{theorem.lift} that $d\pi_{\Hdg}^+$ is the geometric Sen morphism (after restriction to the underlying $v$-sheaf; in particular, $d\pi_{\Hdg}^+$ provides a natural inscribed extension of the geometric Sen morphism which is not, a priori, pinned down by the original geometric Sen morphism because the latter is defined using the completed structure sheaf). 

\subsection{Outline}
In \S\ref{s.filtered-torsors} we recall some constructions with filtered bundles with integrable connection and construct the Kodaira-Spencer morphism. In \S\ref{s.proof} we prove Theorem \ref{theorem.lift}. 

\subsection{Acknowledgements} We thank Tongmu He, Lue Pan, and Juan Esteban Rodriguez Camargo for helpful conversations. This work was supported in part by National Science Foundation grants DMS-2201112 and DMS-2501816.

\section{Filtered bundles with integrable connection}\label{s.filtered-torsors}
Let $L$ be any complete extension of $\mathbb{Q}_p$, and let $S/L$ be a smooth rigid analytic variety. In this section we explain our terminology for filtered vector bundles and torsors with integrable connection, then construct the Kodaira-Spencer morphism. 

\subsection{Filtered vector bundles with integrable connection}
 We consider the exact category of filtered vector bundles with integrable connection on $S$, whose objects are vector bundles $\mathcal{V}$ equipped with a locally separated and exhaustive descending filtration $\Fil^\bullet \mc{V}$ by local direct summands and equipped with a connection $\nabla_{\mathcal{V}}: \mathcal{V} \rightarrow \mathcal{V} \otimes_{\mathcal{O}} \Omega^1_S$ such that $\nabla^2=0$. We say $(\mathcal{V}, \Fil^\bullet, \nabla)$ satisfies Griffiths transversality if $\nabla|_{\Fil^i\mathcal{V}}$ factors through $\Fil^{i-1} \mathcal{V} \otimes_{\mathcal{O}} \Omega^1_S.$ 

\subsection{Filtered torsors with integrable connection}\label{ss.filtered-torsors-with-integrable-connection}
Suppose $G/L$ is a connected linear algebraic group. A filtered $G$-bundle with integrable connection on $S$ is an exact tensor functor $\omega$ from $\Rep_L G$ to filtered vector bundles with integrable connection on $S$. We say it satisfies Griffiths transversality if it factors through the subcategory of filtered vector bundles with integrable connection satisfying Griffiths transversality. 

Given a filtered $G$-bundle with integrable connection $\omega$, we obtain a geometric $G$-torsor $\rho:\mathcal{G}\rightarrow S$ representing the functor of trivializations of the underlying $G$-bundle. In particular, over $\mathcal{G}$ the filtration on the $G$-bundle is transported $G$-equivariantly to a filtration on the trivial bundle that is classified by a $G$-equivariant map $\pi: \mathcal{G} \rightarrow \Fl_G$, where $\Fl_G$ is the flag variety parameterizing trivializations on the trivial $G$-torsor (its geometric connected components are in bijection with conjugacy classes $[\mu]$ of geometric cocharacters of $G$, which are the flag varieties $\Fl_{[\mu]}$ most often considered in the literature). The connection gives rise to a $G$-equivariant splitting $T_{\mathcal{G}/L}= \Lie G \oplus \rho^* T_{S/L}$, where here $\Lie G$ is treated as a vector group and corresponds to the vertical tangent directions via the $G$-action. We refer to the data of $\mathcal{G}$, $\pi$, and the splitting as the associated filtered $G$-torsor with integrable connection; note that one can recover $\omega$ from this data (but to obtain an equivalence between the notions we should further require that the splitting also respect Lie brackets of vector fields). 

\subsection{The Kodaira-Spencer morphism}\label{ss.Kodaira-Spencer}
Continuing with the notation of \S\ref{ss.filtered-torsors-with-integrable-connection}, we can restrict $d\pi$ to $\rho^* T_{S/L} \subseteq T_{\mathcal{G}/L}$ and then descend along the $G$-action back to $S$ to obtain 
\[ \kappa_{\mathcal{G}}: T_{S/L} \rightarrow \omega(\mf{g}) / \Fil^0 \omega(\mf{g}). \]
This is evidently functorial in push-out. When $G=\GL_n$, writing $(\mathcal{V},\Fil^\bullet, \nabla):=\omega(L^n)$,  
\[ \omega(\mf{g})=\mathcal{E}nd(\mathcal{V}) \textrm{ as a vector bundle.} \]
The element $\kappa_{\mathcal{G}}$ can be computed in this case by choosing (locally) another connection $\nabla'$ on $\mathcal{V}$ that does preserve the filtration and then taking $\nabla-\nabla'$ as an element of 
\[\Omega^1_{S} \otimes_{\mathcal{O}_S} \mathcal{E}nd(V)\]
and then passing to
\[\Omega^1_{S} \otimes_{\mathcal{O}_S} (\mathcal{E}nd(V)/\Fil^0\mathcal{E}nd V)\]
to get an element that is independent of the choice of $\nabla'$ (so that, in particular, the local choices all glue to a global section). It follows that $(\mathcal{V}, \Fil^\bullet, \nabla)$ satisfies Griffiths transversality if and only if its Kodaira-Spencer morphism lies in $\gr^{-1}\left(\mathcal{E}nd V\right)$. By the functoriality of the construction, we conclude that in general a filtered $G$-bundle with integrable connection $\omega$ satisfies Griffiths transversality if and only if $\kappa_{\mathcal{G}}$ factors through $\gr^{-1}(\omega(\mf{g}))$. 

\subsection{de Rham local systems}\label{ss.deRham}
Suppose now that $L$ is a $p$-adic field. The results of \cite[\S6]{Scholze.pAdicHodgeTheoryForRigidAnalyticVarieties} show that 
\[ \mathbb{M}: (\mathcal{V}, \Fil^\bullet, \nabla) \rightarrow \Fil^0(\mathcal{V} \otimes_{\mathcal{O}}\mathcal{O}\mathbb{B}_\dR)^{\nabla=0} \]
is a fully faithful exact tensor functor from filtered vector bundles with integrable connection on $S$ satisfying Griffiths transversality to $\mathbb{B}^+_\dR$-local systems on $S_\proet$. Recall that a $\ul{\mathbb{Q}_p}$-local system $\mathbb{L}$ is de Rham if $\mathbb{L}\otimes_{\ul{\mathbb{Q}_p}}\mathbb{B}^+_\dR = \mathbb{M}(\mathcal{V},\Fil^\bullet, \nabla)$, in which case $\mathbb{L} \otimes_{\ul{\mathbb{Q}_p}} \mathcal{O}\mathbb{B}_\dR = \mathcal{V} \otimes_{\mathcal{O}} \mathcal{O}\mathbb{B}_\dR$, compatibly with the natural filtrations and connections on both sides.

\section{Proof of Theorem \ref{theorem.lift}}\label{s.proof}
In this section we prove Theorem \ref{theorem.lift}. Thus let $L$ be a $p$-adic field, let $S/L$ be a smooth rigid analytic variety, let $G/\mathbb{Q}_p$ be a connected linear algebraic group, and let $K \leq G(\mathbb{Q}_p)$ be an open subgroup. We fix a de Rham $K$-torsor $\tilde{S}/S$, and write $\tilde{S}_\dR$ for the associated filtered $G$-torsor with integrable connection. We write $\tilde{\mf{g}}=\tilde{S} \times^K \ul{\mf{g}}$, and $\tilde{\mf{g}}_\dR=\tilde{S}_\dR \times^G \Lie G$, which is the filtered vector bundle (with integrable connection that we will not use) associated to the de Rham local system $\tilde{\mf{g}}$ as in \S\ref{ss.deRham}. On $\tilde{\mf{g}} \otimes \hat{\mc{O}}$ we have the Hodge-Tate filtration, written $\Fil^\bullet_\HT$, and its associated graded $\gr^{\bullet}_\HT$. On $\mf{g}_\dR$ we have the Hodge filtration, written $\Fil^\bullet_\Hdg$, and its associated graded $\gr^{\bullet}_\Hdg$. The geometric Sen morphism associated to $\tilde{S}$ by \cite[Theorem 1.0.4]{RodriguezCamargo.GeometricSenTheoryOverRigidAnalyticSpaces} is written, as in the introduction, as
\[ \kappa_{\tilde{S}}: T_{S/C} \otimes_{\mathcal{O}} \hat{\mathcal{O}} \rightarrow \tilde{\mathfrak{g}}\otimes_{\ul{\mbb{Q}_p}} \hat{\mathcal{O}}(-1), \]
and the Kodaira-Spencer morphism associated to $\tilde{S}_\dR$ (see \S\ref{ss.filtered-torsors-with-integrable-connection}) is written
\[ \kappa_{\tilde{S}_\dR}: T_{S/C} \rightarrow  \gr^{-1}_\Hdg (\tilde{\mf{g}}_\dR). \]

We have a natural map
\begin{equation}\label{eq.nat-surj-body}q_{\tilde{S}}: \left(\Fil^1_\HT(\tilde{\mf{g}} \otimes_{\ul{\mbb{Q}_p}} \hat{\mathcal{O}})\right)(-1) \twoheadrightarrow \left(\gr^1_\HT(\tilde{\mf{g}} \otimes_{\ul{\mbb{Q}_p}} \hat{\mathcal{O}})\right)(-1) = \left(\gr^{-1}_\Hdg (\tilde{\mf{g}}_\dR)\right) \otimes_{\mathcal{O}} \hat{\mathcal{O}}, \end{equation}
where the equality is by the Hodge-Tate comparison. We need to show that $\kappa_{\tilde{S}}$ factors through $\left(\Fil^1_\HT(\tilde{\mf{g}} \otimes_{\ul{\mbb{Q}_p}} \hat{\mathcal{O}})\right)(-1)$, that $\kappa_{\tilde{S}_\dR}=q_{\tilde{S}} \circ \kappa_{\tilde{S}}$, and that $\kappa_{\tilde{S}}$ is the unique such lift.

\subsection{Uniqueness of the lift}
Consider the short exact sequence on $S_\proet$
\[ 0 \rightarrow \Fil^2_\HT(\tilde{\mf{g}} \otimes_{\ul{\mbb{Q}_p}} \hat{\mathcal{O}}) \rightarrow \Fil^1_\HT(\tilde{\mf{g}} \otimes_{\ul{\mbb{Q}_p}} \hat{\mathcal{O}}) \rightarrow \gr^{1}_\HT(\tilde{\mf{g}} \otimes_{\ul{\mbb{Q}_p}} \hat{\mathcal{O}}) \rightarrow 0 \]
Since $\Omega_{S/C} \otimes_{\mathcal{O}} \hat{\mathcal{O}}(-1)$ is a locally free $\hat{\mathcal{O}}$-module, the sequence remains exact when we tensor with it, so we obtain a short exact sequence
\[ 0 \rightarrow \Omega_{S/C} \otimes_{\mathcal{O}}  \Fil^2_\HT(\tilde{\mf{g}} \otimes_{\ul{\mbb{Q}_p}}\hat{\mathcal{O}})(-1) \rightarrow \Omega_{S/C} \otimes_{\mathcal{O}}  \Fil^1_\HT(\tilde{\mf{g}} \otimes_{\ul{\mbb{Q}_p}} \hat{\mathcal{O}}) \rightarrow \Omega_{S/C} \otimes_{\mathcal{O}}  \gr^{1}_\HT(\tilde{\mf{g}} \otimes_{\ul{\mbb{Q}_p}} \hat{\mathcal{O}}) \rightarrow 0. \]
The Kodaira-Spencer map is a section over $S$ of the last non-zero term, while the geometric Sen morphism is a section over $S$ of the middle non-zero term, and the map between them is precisely that which should be considered for the lift in Theorem \ref{theorem.lift}. To show uniqueness of a lift, it thus suffices to show that, for $\nu: S_\proet \rightarrow S_\et$ the natural morphism, that 
\[ \nu_* \left(\Omega_{S/C} \otimes_{\mathcal{O}}  \Fil^2_\HT(\tilde{\mf{g}} \otimes_{\ul{\mbb{Q}_p}}\hat{\mathcal{O}})(-1)\right)=0.\]

By the projection formula, it suffices to prove $\nu_* \left(\Fil^2_\HT(\tilde{\mf{g}} \otimes_{\ul{\mbb{Q}_p}}\hat{\mathcal{O}})(-1)\right)=0$. To that end, note that for $k \geq 2$ we have the exact sequence  
\[ 0 \rightarrow \left(\Fil^{k+1}_\HT(\tilde{\mf{g}} \otimes_{\ul{\mbb{Q}_p}} \hat{\mathcal{O}})\right)(-1) \rightarrow \left(\Fil^{k}_\HT(\tilde{\mf{g}} \otimes_{\ul{\mbb{Q}_p}} \hat{\mathcal{O}})\right)(-1) \rightarrow \left(\gr^{k}_\HT(\tilde{\mf{g}} \otimes_{\ul{\mbb{Q}_p}} \hat{\mathcal{O}})\right)(-1) \rightarrow 0. \]
It thus suffices to show that, for all $k \geq 2$, 
\begin{equation}\label{eq.uniqueness-gr-vanishing}\nu_*\left(\left(\gr^{k}_\HT(\tilde{\mf{g}} \otimes_{\ul{\mbb{Q}_p}} \hat{\mathcal{O}})\right)(-1)\right)=0.\end{equation}
Indeed, we will then have $\nu_*\left(\left(\Fil^{k}_\HT(\tilde{\mf{g}} \otimes_{\ul{\mbb{Q}_p}} \hat{\mathcal{O}})\right)(-1)\right) = \nu_*\left(\left(\Fil^{N}_\HT(\tilde{\mf{g}} \otimes_{\ul{\mbb{Q}_p}} \hat{\mathcal{O}})\right)(-1)\right)$ for any $N \geq 2$, but since the filtration is exhaustive the latter can be chosen large enough to be zero. Now, by the Hodge-Tate comparison, 
\[ \gr^{k}_\HT(\tilde{\mf{g}} \otimes_{\ul{\mbb{Q}_p}} \hat{\mathcal{O}}) = \gr^{-k}_\Hdg(\tilde{\mf{g}}_\dR)\otimes_{\mathcal{O}} \hat{\mathcal{O}}(k)\]
so, by the projection formula, to establish \eqref{eq.uniqueness-gr-vanishing} it suffices to observe that, since $k-1 \geq 1$, \cite[Corollary 6.19]{Scholze.pAdicHodgeTheoryForRigidAnalyticVarieties} gives $\nu_*\left( \hat{\mathcal{O}}(k-1)\right)=0$. This establishes the uniqueness in Theorem \ref{theorem.lift}.

\subsection{Computation of the geometric Sen morphism}

Because both the geometric Sen morphism and Kodaira-Spencer morphism are functorial in push-out (the former by the final part of \cite[Theorem 1.0.4]{RodriguezCamargo.GeometricSenTheoryOverRigidAnalyticSpaces} and the latter by the discussion of \S\ref{ss.Kodaira-Spencer}), to show that $\kappa_{\tilde{S}}$ factors through $\left(\Fil^1_\HT(\tilde{\mf{g}} \otimes_{\ul{\mbb{Q}_p}} \hat{\mathcal{O}})\right)(-1)$ and that $\kappa_{\tilde{S}_\dR}=q_{\tilde{S}} \circ \kappa_{\tilde{S}}$, we may push-out along a faithful representation of $G$ to assume that $G=\GL_n$ and that $K=\GL_n(\mathbb{Q}_p)$. The following lemma establishes the compatibility in this case, concluding the proof of Theorem \ref{theorem.lift}. 

\begin{lemma}
    Let $\mathbb{L}$ be a de Rham $\mathbb{Q}_p$-local system on $S$ and let $(\mc{V},\nabla, \Fil^\bullet)$ be the associated filtered vector bundle with integrable connection on $S$. The geometric Sen morphism 
    \begin{equation}\label{eq.deRhamGeometricSenCompEq1} \kappa_{\mathbb{L}}: T_S \xrightarrow{} \mc{E}nd_{\hat{\mc{O}}}(\mbb{L} \otimes_{\ul{\mbb{Q}_p}} \hat{\mc{O}})(-1) \end{equation}
factors through $\Fil^1_\HT \left(\mc{E}nd_{\hat{\mc{O}}}(\mbb{L} \hat{\otimes} \hat{\mc{O}})\right)$ and is a lift of the composition
\begin{equation}\label{eq.deRhamGeometricSenCompEq2} \kappa_{(\mc{V},\nabla, \Fil^\bullet)}: T_S   \rightarrow \gr^{-1}_{\Hdg}\left(\mc{E}nd_{\mc{O}}(\mc{V})\right)\otimes_{\mc{O}} \hat{\mc{O}}=\gr^1_{\HT}(\mc{E}nd(\mathbb{L} \otimes \hat{\mc{O}}))(-1).\end{equation}
\end{lemma}
\begin{proof}
We write $\Omega^1_S(-1):=\Omega^1 \otimes_{\mathcal{O}} \hat{\mathcal{O}}(-1)$. 
Let $\theta_\mathbb{L}$ be the map $\mathbb{L} \otimes_{\ul{\mbb{Q}_p}} \hat{\mc{O}} \rightarrow \mathbb{L} \otimes_{\ul{\mbb{Q}_p}} \Omega^1_S(-1)$ as in \cite[Theorem 1.0.3]{RodriguezCamargo.GeometricSenTheoryOverRigidAnalyticSpaces} (so that $\kappa_{\mathbb{L}}$ is obtained by contracting $\theta_\mathbb{L}$ with sections of $T_S$).  Similarly, let $\theta_{\mathcal{V}}$ be the  the section of 
\[ \gr^{-1}_\Hdg\left(\mc{E}nd(\mc{V})\right) \otimes_{\mathcal{O}} \Omega^1_S=\bigoplus_i \mc{H}om(\gr^{i}_\Hdg\mc{V}, \gr^{i-1}_\Hdg\mc{V}) \otimes \Omega^1_S \]
such that the map of \eqref{eq.deRhamGeometricSenCompEq2} is obtained by contraction with $\theta_{\mc{V}}$.
In degree $i$, $\theta_{\mc{V}}$ is the $\mc{O}$-linear map induced by
\[ \Fil^i_\Hdg \mc{V} \xrightarrow{\nabla} \Fil^{i-1} \mc{V}\otimes \Omega^1_S \rightarrow \gr^{i-1}_\Hdg \mc{V} \otimes \Omega^1_S. \]

It thus suffices to see that, after applying the Hodge-Tate comparison, $\theta_{\mathbb{L}}$ is a lift of $\theta_{\mathcal{V}}$. We will accomplish this by spelling out how the de Rham comparison gives rise to the Hodge-Tate comparison and then using the computation of the geometric Sen operator on $\mathcal{O}\mathbb{C}=\gr^{0}\mc{O}\mbb{B}_\dR$ from \cite[\S 3.5]{RodriguezCamargo.GeometricSenTheoryOverRigidAnalyticSpaces}. 

Thus we first recall that by the de Rham comparison we have 
\[ \mbb{L} \otimes_{\ul{\mbb{Q}_p}} \mc{O}\mbb{B}_\dR = \mc{V} \otimes_{\mc{O}} \mc{O}\mbb{B}_\dR. \]
This identification matches the connection $1 \otimes \nabla$ on the left with $\nabla \otimes 1 + 1 \otimes \nabla$ on the right. It matches  the filtration $\mathbb{L} \otimes \Fil^\bullet \mc{O}\mbb{B}_\dR$ on the right with the convolved filtration $\Fil^\bullet_\Hdg \otimes \Fil^\bullet \mc{O}\mbb{B}_\dR$ on the right. This extends to an identification of de Rham complexes 
\[ \mbb{L} \otimes_{\ul{\mbb{Q}_p}} (\mc{O}\mbb{B}_\dR \otimes_{\mc{O}} \Omega^\bullet) = (\mc{V} \otimes_{\mathcal{O}} \Omega^\bullet_S) \otimes_{\mc{O}} \mc{O}\mbb{B}_\dR. \]
Note that these de Rham complexes are filtered exact, thus, taking the zeroth graded piece, we can compute $\tilde{\mathbb{L}} \otimes_{\ul{\mbb{Q}_p}}\hat{\mc{O}}$ via the following truncated diagram, whose columns are exact (note that $\gr^{i}\mc{O}\mbb{B}_\dR=\mc{O}\mbb{C}(i)$):
\[\begin{tikzcd}
	0 & 0 \\
	{\mathrm{Ker} (\theta_V\otimes 1 + 1\otimes \overline{\nabla})} & {\mathbb{L} \otimes_{\ul{\mbb{Q}_p}} \hat{\mc{O}} } \\
	{\bigoplus_j \gr^{-j}_{\Hdg}\mc{V} \otimes_{\mc{O}} \mc{O}\mbb{C}(j)} & {\mathbb{L} \otimes_{\ul{\mbb{Q}_p}} \mc{O}\mbb{C} } \\
	{\bigoplus_j \gr^{-j}_{\Hdg}\mc{V} \otimes_{\mc{O}} \mc{O}\mbb{C}(j) \otimes_{\mc{O}} \Omega^1} & {\mathbb{L} \otimes_{\ul{\mbb{Q}_p}} \mc{O}\mbb{C} \otimes_{\mc{O}} \Omega}
	\arrow["{=}"{description}, from=1-1, to=1-2]
	\arrow[from=1-1, to=2-1]
	\arrow[from=1-2, to=2-2]
	\arrow["{=}"{description}, no head, from=2-1, to=2-2]
	\arrow[from=2-1, to=3-1]
	\arrow[from=2-2, to=3-2]
	\arrow["{\theta_{\mc{V}}\otimes 1 + 1\otimes \overline{\nabla}}", from=3-1, to=4-1]
	\arrow["{=}"{description}, no head, from=3-2, to=3-1]
	\arrow["{1\otimes\overline{\nabla}}", from=3-2, to=4-2]
	\arrow["{=}"{description}, no head, from=4-2, to=4-1]
\end{tikzcd}\]
Now, note that $\theta_{\mathcal{V}}\otimes 1 + 1 \otimes \overline{\nabla}$ preserves the filtration 
\[ \Fil^i_{\HT}(\oplus_{j} \ldots)=\oplus_{j \geq i}\ldots. \]
The induced filtration on the kernel is the Hodge-Tate filtration on $\mathbb{L} \otimes_{\ul{\mbb{Q}_p}} \hat{\mc{O}}$, and because $\theta_\mathcal{V}$ increases the $j$ degree, the associated graded is computed as 
\[ \mr{Ker} \left( \oplus_j \gr^{-j}_{\Hdg}\mc{V} \otimes_{\mc{O}} \mc{O}\mbb{C}(j)) \xrightarrow{1 \otimes \overline{\nabla}} \oplus_j \gr^{-j}_{\Hdg}\mc{V} \otimes_{\mc{O}} \mc{O}\mbb{C}(j)\otimes_{\mc{O}} \Omega \right) = \oplus_j \gr^{-j}_{\Hdg}\mc{V}\otimes_{\mc{O}}\hat{\mc{O}}(j)\]
and this computation \emph{is} the Hodge-Tate comparison isomorphism for $\mathbb{L}$. 

We can also use this diagram to compute $\theta_\mathbb{L}$. Indeed, by \cite[Proposition 3.5.2]{RodriguezCamargo.GeometricSenTheoryOverRigidAnalyticSpaces}, the geometric Sen operator on $\mc{O}\mbb{C}$ is $- \overline{\nabla}$. Tate twists do not change the geometric Sen operator, nor does tensoring with an \'{e}tale vector bundle, so the geometric Sen operator on 
\[ \bigoplus \gr_\Hdg^{-j}\mathcal{V} \otimes_{\mc{O}} \mc{O}(\mbb{C})(j)\]
is just $-1\otimes \overline{\nabla}$. On the kernel of $\theta_{\mathcal{V}} \otimes 1  + 1 \otimes \overline{\nabla}$ we have $\theta_{\mathcal{V}} \otimes 1 = -1 \otimes \overline{\nabla}$, and thus on this kernel, which is $\mathbb{L} \otimes_{\ul{\mathbb{Q}_p}} \hat{\mathcal{O}}$, 
the geometric Sen operator agrees with $\theta_{\mathcal{V}} \otimes 1$. In particular, it increases the Hodge-Tate filtration by $1$, and the image in the associated graded of the endomorphism sheaf is precisely the map induced by $\theta_\mathcal{V}$ and the Hodge-Tate comparison. 
\end{proof}

\bibliographystyle{plain}
\bibliography{references, preprints}

\end{document}